\documentclass[11pt, reqno]{article}
\usepackage{amsmath,amssymb,amsfonts,amsthm}
\usepackage{color}

\usepackage[colorlinks=true,linkcolor=blue,citecolor=black,urlcolor=black,pdfborder={0 0 0}]{hyperref}
\usepackage{graphicx}
\usepackage{cleveref}
\usepackage{esvect}

\usepackage{bbm}
\usepackage{enumitem}

\usepackage{caption}
\usepackage{thmtools}
\usepackage{commath}
\usepackage{mathrsfs}

\usepackage{relsize}
\usepackage{enumitem}
\usepackage{eufrak}

\crefname{theorem}{Theorem}{Theorems}
\crefname{thm}{Theorem}{Theorems}
\crefname{lemma}{Lemma}{Lemmas}
\crefname{lem}{Lemma}{Lemmas}
\crefname{remark}{Remark}{Remarks}
\crefname{prop}{Proposition}{Propositions}
\crefname{defn}{Definition}{Definitions}
\crefname{corollary}{Corollary}{Corollaries}
\crefname{conjecture}{Conjecture}{Conjectures}
\crefname{question}{Question}{Questions}
\crefname{chapter}{Chapter}{Chapters}
\crefname{section}{Section}{Sections}
\crefname{figure}{Figure}{Figures}

\theoremstyle{plain}
\newtheorem{thm}{Theorem}[section]
\newtheorem{lemma}[thm]{Lemma}

\newtheorem{corollary}[thm]{Corollary}
\newtheorem{prop}[thm]{Proposition}

\newtheorem{question}[thm]{Question}
\theoremstyle{definition}

\theoremstyle{remark}
\newtheorem*{remark}{Remark}

\numberwithin{equation}{section}

\renewcommand{\P}{\mathbb P}
\newcommand{\E}{\mathbb E}

\newcommand{\R}{\mathbb R}
\newcommand{\Z}{\mathbb Z}
\newcommand{\N}{\mathbb N}

\newcommand{\sA}{\mathscr A}

\newcommand{\sF}{\mathscr F}

\newcommand{\WUSF}{\mathsf{WUSF}}
\newcommand{\FUSF}{\mathsf{FUSF}}

\usepackage[margin=1.05in]{geometry}
\usepackage{extarrows}
\usepackage{commath}
\usepackage{mathtools}
\newcommand{\eps}{\varepsilon}
\usepackage{bbm}
\usepackage{setspace}
\usepackage{enumitem}
\usepackage{tikz}
\usepackage[numbers,sort]{natbib}
\usepackage{subcaption}

\def\bmid{\,\big|\,}
\def\Bmid{\;\Big|\;}

\newcommand{\freq}{\operatorname{freq}}

\newcommand{\cost}{\operatorname{cost}}
\newcommand{\spann}{\mathcal{S}}
\newcommand{\conn}{\mathcal{U}}

\def\lora{\longrightarrow}
\def\llra{\longleftrightarrow}

\title{Kazhdan groups have cost $1$}

\author{Tom Hutchcroft and G\'abor Pete}

\AtEndDocument{
  \bigskip
  \small
  \par
  \textsc{Tom Hutchcroft} \par
  \textsc{Statslab, DPMMS, University of Cambridge} \par
  \textit{Email:} \texttt{t.hutchcroft@maths.cam.ac.uk} \par

  \medskip

  \textsc{G\'abor Pete} \par
  \textsc{Alfr\'ed R\'enyi Institute of Mathematics, and} \par
  \textsc{Mathematical Institute, Budapest University of Technology and Economics} \par
  \textit{Email:} \texttt{gabor@math.bme.hu}\par
}

\begin{document}

\maketitle

\begin{abstract}
We prove that every countably infinite group with Kazhdan's property (T) has cost $1$, 
answering a well-known question of Gaboriau. It remains open if they have fixed price $1$.
\end{abstract}

\section{Introduction}

The \emph{cost} of a free, probability measure preserving (p.m.p.)~action of a group is an orbit-equivalence invariant that was introduced by Levitt \cite{Levitt} and studied  extensively by Gaboriau \cite{Gabcout,MR1953191,GabICM}. Gaboriau used the notion of cost to prove several remarkable theorems, including that free groups of different ranks cannot have orbit equivalent free, ergodic,  p.m.p.\ actions. This result is in stark contrast with the amenable case, in which Ornstein and Weiss \cite{OW80} proved that any two free, ergodic p.m.p.\ actions are orbit equivalent. These results sparked a surge of interest in the cost of group actions, the fruits of which are summarised in the monographs and surveys \cite{MR2095154,MR2583950,GabICM,MR2807836}.

The cost of a \emph{group} is defined to be the infimal cost of all free, ergodic p.m.p.~actions of the group.
 We will employ here the following probabilistic definition, which is shown to be equivalent to the classical definition in \cite[Proposition 29.5]{MR2095154}. 
Let $\Gamma$ be a countable group. We define 
$\spann(\Gamma)$ to be the set of \textbf{connected spanning graphs} on $\Gamma$, that is, the set of connected, undirected, simple graphs with vertex set $\Gamma$. Formally, we define $\spann(\Gamma)$ to be the set of $\omega \in \{0,1\}^{\Gamma \times \Gamma}$ such that $\omega(a,b)=\omega(b,a)$ for every $a,b \in \Gamma$ and such that for each $a,b \in \Gamma$ there exists $n\geq 0$ and a sequence $a=a_0,a_1,\ldots,a_n =b$ in $\Gamma$ such that $\omega(a_{i-1},a_{i})=1$ for every $1\leq i \leq n$.  We equip $\{0,1\}^{\Gamma \times \Gamma}$ with the product topology and associated Borel $\sigma$-algebra, and equip $\spann(\Gamma)$ with the subspace topology and Borel $\sigma$-algebra. Note that $\spann(\Gamma)$ is \emph{not} closed in $\{0,1\}^{\Gamma\times\Gamma}$ when $\Gamma$ is infinite. 
For each $\omega \in \spann(\Gamma)$ and $\gamma \in \Gamma$ we define $\gamma \omega$ by setting $\gamma \omega(u,v) = \omega(\gamma^{-1} u, \gamma^{-1} v)$. We say that a probability measure on $\spann(\Gamma)$ is $\Gamma$\textbf{-invariant} if $\mu(\sA)= \mu(\gamma^{-1} \sA)$ for every Borel set $\sA \subseteq \spann(\Gamma)$, and write $M(\Gamma,\spann(\Gamma))$ for the set of $\Gamma$-invariant probability measures on $\spann(\Gamma)$.  
The {\bf cost of the group} $\Gamma$ can be defined to be
\begin{equation}
\label{eq:cost_def}
\cost(\Gamma) = \frac{1}{2} \inf \left\{ \int_{\omega \in \spann(\Gamma)} \deg_\omega(o) \dif \mu(\omega) : \mu \in M(\Gamma,\spann(\Gamma)) \right\},
\end{equation}
where $o$ is the identity element of $\Gamma$ and $\deg_\omega(o)$ is the degree of $o$ in the graph $\omega \in \spann(\Gamma)$.
Note that for nonamenable groups with cost $1$, and more generally for any non-treeable group, the infimum in \eqref{eq:cost_def} is not attained \cite[Propositions 30.4 and 30.6]{MR2095154}.

Every countably infinite amenable group has cost 1 by Orstein-Weiss \cite{OW80} (see \cite[Section 5]{BLPS99} for a probabilistic proof), while the free group $\mathbb{F}_k$ has cost $k$ \cite{MR1646912}. There are also however many nonamenable groups with cost 1, including the direct product $\Gamma_1 \times \Gamma_2$ of any two countably infinite groups $\Gamma_1$ and $\Gamma_2$ \cite[Theorem 33.3]{MR2095154} (which is nonamenable if at least one of $\Gamma_1$ or $\Gamma_2$ is nonamenable), and $\mathrm{SL}_d(\Z)$ with $d\ge 3$ \cite{Gabcout}.  See \cite{MR1646912} for many further examples.

In general, computing the cost of a group is not an easy task.  Nevertheless, one possible approach is suggested by the following question of Gaboriau, which connects the cost  to the \emph{first $\ell^2$-Betti number $\beta_1(\Gamma)$} of the group. This is a measure-equivalence invariant of the group that can be defined to be  the von Neumann dimension of the space of harmonic Dirichlet functions of any Cayley graph of the group. Equivalently, $\beta_1(\Gamma)$ can be defined in terms of the expected degree of the \emph{free uniform spanning forest} in any Cayley graph of $\Gamma$ by the equality $\E\deg_{\FUSF}(o)=2+2\beta_1(\Gamma)$, see \cite[Section 10.8]{LP:book}. Gaboriau \cite{MR1953191} proved that $\cost(\Gamma) \geq 1 + \beta_1(\Gamma)$ and asked whether this inequality is ever strict.

\begin{question}[Gaboriau]\label{q.costbeta}
Is $\cost(\Gamma) = 1 + \beta_1(\Gamma)$ for every countably infinite group $\Gamma$?
\end{question}

For groups with Kazhdan's property (T), defined below, it was proven by Bekka and Valette  that $\beta_1=0$ \cite{BV97}. However, in spite of several works connecting property (T), cost, and percolation theory \cite{GabOHD,LS99,IKT09, LyonsFixed}, the cost of Kazhdan groups has remained elusive \cite[Question 6.4]{GabICM}, and has thus become a famous test example for Question~\ref{q.costbeta}. This paper addresses this question.

\begin{thm}
\label{thm:main}
Let $\Gamma$ be a countably infinite Kazhdan group. 
 Then $\Gamma$ has cost $1$.
\end{thm}

In fact, our proof gives slightly more: It is a classical result of Kazhdan that every countable Kazhdan group is finitely generated \cite[Theorem 1.3.1]{MR2415834}. The proof of \cref{thm:main}  establishes that for every $\eps>0$ and every finite symmetric generating set $S$ of $\Gamma$, there is a $\Gamma$-invariant measure on connected, spanning {\it subgraphs} of the associated Cayley graph with average degree at most $2+\eps$.   

Our proof will apply the following probabilistic characterization of property (T) due to Glasner and Weiss \cite{MR1475550}, which the reader may take  as the definition of property (T) for the purposes of this paper. 
Let $\Gamma$ be a countable group, and let $\Gamma \curvearrowright X$ be an action of $\Gamma$ by homeomorphisms  on a topological space $X$. We write 
$M(\Gamma,X)$ for the space of $\Gamma$-invariant Borel probability measures on $X$, which is equipped with the weak$^*$ topology, and write $E(\Gamma,X) \subseteq M(\Gamma,X)$ for the subspace of \emph{ergodic} $\Gamma$-invariant Borel probability measures on $X$. Here, we recall that an event $\sA \subseteq X$ is said to be \textbf{invariant} if $\gamma \sA = \sA$ for every $\gamma \in \Gamma$, and that a measure $\mu\in M(\Gamma,X)$ is said to be \textbf{ergodic} if $\mu(\sA) \in \{0,1\}$ for every invariant event $\sA$.

\begin{thm}[Glasner and Weiss 1997]\label{t.GW}
Let $\Gamma$ be a countably infinite group, and consider the natural action of $\Gamma$ on $\Omega=\{0,1\}^\Gamma$. Then the following are equivalent.
\begin{enumerate}
	\item $\Gamma$ has Kazhdan's property (T).
	\item $E(\Gamma,\Omega)$ is closed in $M(\Gamma,\Omega)$.
	\item $E(\Gamma,\Omega)$ is not dense in $M(\Gamma,\Omega)$.
\end{enumerate}
\end{thm}

See e.g.\ \cite{MR2415834} for further background on Kazhdan groups.


It remains open if Kazhdan groups have \emph{fixed price $1$}, i.e., if \emph{every} free ergodic p.m.p.\ action has cost $1$. (In contrast, \cref{thm:main}  implies that there \emph{exists} a free, ergodic p.m.p.\ action with cost $1$, see \cite[Proposition 29.1]{MR2095154}.)
 Ab\'ert and Weiss \cite{AbW} proved that Bernoulli actions have maximal cost among all free ergodic p.m.p.\ actions of a given group, and probabilistically this means that the maximal cost of the free ergodic p.m.p.\ actions of  a countable group $\Gamma$ is equal to
\begin{equation}
\label{eq:upper_cost_def}
\cost^*(\Gamma) = \frac{1}{2} \inf \left\{ \int_{\omega \in \spann(\Gamma)} \deg_\omega(o) \dif \mu(\omega) : \mu \in F_{\mathrm{IID}}(\Gamma,\spann(\Gamma)) \right\},
\end{equation}
where $F_{\mathrm{IID}}(\Gamma,\spann(\Gamma))\subseteq M(\Gamma,\spann(\Gamma))$ is the set of $\Gamma$-invariant measures on $S(\Gamma)$ that arise as \emph{factors of i.i.d.}\ processes on $\Gamma$.
Our construction is very far from being a factor of i.i.d., and therefore seems unsuitable to study $\cost^*(\Gamma)$. See Remark~\ref{r.fiid} for further discussion. The question of fixed price $1$ for Kazhdan groups is of particular interest due to its connection to the Ab\'ert-Nikolov rank gradient conjecture \cite[Conjecture 17]{MR2966663}.


An extension of our results to groups with \emph{relative} property (T) is sketched in \cref{s.rel}.

\section{Proof}

\subsection{A reduction}\label{ss.redu}

We begin our proof with the following proposition, which shows that it suffices for us to find sparse random graphs on $\Gamma$ that have a unique infinite connected component. 
We define $\conn(\Gamma) \subseteq \{0,1\}^{\Gamma \times \Gamma}$ to be the set of graphs on $\Gamma$ that have a unique infinite connected component. 

\begin{prop}
\label{prop:sparse_cost}
Let $\Gamma$ be an infinite, finitely generated group. Then
\[
\cost(\Gamma) \leq 1 + \frac{1}{2} \inf \left\{ \int_{\omega \in \conn(\Gamma)} \deg_\omega(o) \dif \mu(\omega) : \mu \in M(\Gamma,\conn(\Gamma)) \right\}.
\]
\end{prop}

\cref{prop:sparse_cost} can be easily deduced from the \emph{induction formula} of Gaboriau \cite[Proposition II.6]{Gabcout}. We provide a direct proof for completeness.

\begin{proof} 
 Take a Cayley graph $G$ corresponding to a finite symmetric generating set of $\Gamma$. 
Let $\mu \in M(\Gamma,\conn(\Gamma))$, let $\omega$ be a random variable with law $\mu$, and let $\eta_0$ be the set of vertices of its unique infinite connected component.  For each $i\ge 1$, let $\eta_i$ be the set of vertices in $G$ that have graph distance exactly $i$ from $\eta_0$ in $G$. Note that $\bigcup_{i\geq 0} \eta_i = \Gamma$, and that if $i\geq 1$ then every vertex in $\eta_i$ has at least one neighbour in $\eta_{i-1}$. 
For each $i\geq 1$ and each vertex $v \in \eta_i$, let $e^\rightarrow(v)$ be chosen uniformly at random from among those oriented edges of $G$ that begin at $v$ and end at a vertex of $\eta_{i-1}$, and let $e(v)$ be the \emph{unoriented} edge obtained by forgetting the orientation of $e^\rightarrow(v)$. These choices are made independently conditional on $\omega$. We define $\zeta =\{e(v) : v\in V\setminus \eta_0\}$ and define $\nu$ to be the law of $\xi=\omega \cup \zeta$. 
 We clearly have that $\xi$ is in $\spann(\Gamma)$ whenever $\omega \in \conn(\Gamma)$, and hence that $\nu \in M(\Gamma,\spann(\Gamma))$. 
 On the other hand, the mass-transport principle (see \cite[Section 8.1]{LP:book}) implies that, writing $\P$ and $\E$ for probabilities and expectations taken with respect to the joint law of $\omega$ and $\{e(v) : v \in V \setminus \eta\}$,
\[
\E \deg_\zeta(o) = \P(o\notin \eta_0) + \E \sum_{v\in V} \mathbbm{1}\bigl(v\notin \eta_0,\, e^\rightarrow(v)^+\!=o\bigr)  = 2\P(o\notin \eta_0) \leq 2,
\]
where $e^\rightarrow(v)^+$ denotes the other endpoint of $e^\rightarrow(v)$.
We deduce that
\[
\int_{\xi \in \spann(\Gamma)} \deg_\xi(o) \dif \nu(\xi) = \E \deg_\zeta(o)+ \E \deg_\omega(o)  \leq 2 + \int_{\omega \in \conn(\Gamma)} \deg_\omega(o) \dif \mu(\omega),
\]
and the claim follows by taking the infimum over $\mu \in M(\Gamma,\conn(\Gamma))$.
\end{proof}

\begin{remark}\label{r.WUSF} An arguably more canonical way to prove \cref{prop:sparse_cost} is to take the union of $\omega$ with an independent copy of the \emph{wired uniform spanning forest} ($\WUSF$) of the Cayley graph $G$. Indeed, it is clear that \emph{some} components of $\WUSF$ must intersect the infinite component of $\omega$ a.s., and it follows by indistinguishability of trees in $\WUSF$ \cite{HN15a} that \emph{every} tree intersects the infinite component of $\omega$ a.s., so that the union of $\WUSF$ with $\omega$ is a.s.\ connected. (It should also be possible to argue that this union is connected more directly, using Wilson's algorithm \cite{Wilson96,BLPS}.) The result then follows since $\WUSF$ has expected degree $2$ in any transitive graph \cite[Theorem 6.4]{BLPS}.
 
This alternative construction may be of interest for the following reason: It is well known \cite[Question 10.12]{LP:book} that an affirmative answer to Question~\ref{q.costbeta} would follow if one could construct for every $\eps>0$ an invariant coupling $(\FUSF,\eta)$ of the \emph{free uniform spanning forest} of a Cayley graph of $\Gamma$ with a percolation process $\eta$ of density at most $\eps$ such that $\FUSF\cup\eta \in S(\Gamma)$ almost surely. Since Kazhdan groups have $\beta_1=0$, their free and wired uniform spanning forests always coincide \cite[Section 10.2]{LP:book}, so that proving \cref{thm:main} via this alternative proof of Proposition~\ref{prop:sparse_cost} can be seen as a realization of this possibly general strategy.
\end{remark}

\subsection{A construction}\label{ss.constru}

We now construct an invariant measure $\mu \in M(\Gamma,\conn(\Gamma))$ with arbitrarily small expected degree. We will work on an arbitrary Cayley graph of the Kazhdan group $\Gamma$, and the measure we construct will be concentrated on subgraphs of this Cayley graph. (Recall from the introduction that countable Kazhdan groups are always finitely generated.)

Let $G=(V,E)$ be a connected, locally finite graph. 
For each $\omega\in \{0,1\}^V$, the \textbf{clusters} of $\omega$ are defined to be the vertex sets of the connected components of the subgraph of $G$ induced by the vertex set $\{v\in V: \omega(v)=1\}$ (that is, the subgraph of $G$ with vertex set $\{v\in V: \omega(v)=1\}$ and containing every edge of $G$ both of whose endpoints belong to this set). 
Fix $p\in (0,1)$, and let $\mu_1$ be the law of Bernoulli-$p$ site percolation on $G$. For each $i\geq 1$, we recursively define $\mu_{i+1}$ to be the law of the random configuration $\omega \in \{0,1\}^V$ obtained as follows:
\begin{enumerate}
	\item Let $\omega_1,\omega_2\in \{0,1\}^V$ be independent random variables each with law $\mu_i$.
	\item Let $\eta_1$ and $\eta_2$ be obtained from $\omega_1$ and $\omega_2$ respectively by choosing to either delete or retain each cluster independently at random
 with retention probability 
	\[ q(p) := \frac{1-\sqrt{1-p}}{p} \in \left(\frac{1}{2},1\right).\]
	\item Let $\omega$ be the union of the configurations $\eta_1$ and $\eta_2$.
\end{enumerate}
It follows by  induction that if $G$ is a Cayley graph of a finitely generated group $\Gamma$ then $\mu_i \in M(\Gamma,\Omega)$ for every $i\geq 1$. More generally, for each measure $\mu$ on $\{0,1\}^V$ and $q\in [0,1]$ we write $\mu^q$ for the \textbf{$q$-thinned} measure, which is the law of the random variable $\eta$ obtained by taking a random variable $\omega$ with law $\mu$ and choosing to either delete or retain each cluster of $\omega$ independently at random with retention probability $q$. (See \cite[Section 6]{LS99} for a more formal construction of this measure.)

We write $\delta_V$ and $\delta_\emptyset$ for the probability measures on $\{0,1\}^V$ giving all their mass to the all $1$ and all $0$ configurations respectively.

\begin{prop}
\label{cor:degeneration}
Let $G=(V,E)$ be a connected, locally finite graph, let $p\in (0,1)$ and let $(\mu_i)_{i\geq 1}$ be as above. Then $\mu_i(\{\omega: \omega(u) =1\})=p$ for every $i\geq 1$ and $u\in V$ and 
$\mu_i$ weak$^*$ converges to the measure $p \delta_V + (1-p) \delta_\emptyset$ as $i\to\infty$.
\end{prop}


\begin{proof}
It suffices to prove that for every pair of adjacent vertices $u,v \in V$ we have that
\[
\mu_i(\{\omega: \omega(u) =1\})=p
\quad \text{ for every $i\geq 1$ } \quad \text{ and } \quad 
\lim_{i\to\infty}\mu_i\bigl(\{\omega: \omega(u)=\omega(v)\}\bigr) = 1.
\]
For each $u,v\in V$ and $i\geq 1$ let $p_i(u) = \mu_i(\{\omega: \omega(u) =1\})$ and let $\sigma_i(u,v) = \mu_i(\{\omega: \omega(u) = \omega(v)=1\})$. Note that $p_1(u)=p$ for every $u\in V$, that $\sigma_1(u,v) = p^2 >0$ for every $u,v\in V$, and that $\sigma_i(u,v) \leq p_i(u)$ for every $u,v\in V$ and $i\geq 1$.
 Write $q=q(p)$. 
For each $i\geq 1$ and $u\in V$, it follows by definition of $\mu_{i+1}$ that
\begin{equation}
p_{i+1}(u)\\=(1-(1-q)^2) \, p_i(u)^2 + 2q \, p_i(u) \, (1-p_i(u))\\ = \phi\big( p_i(u) \big),
\end{equation}
where $\phi:\R\lora \R$ is the polynomial
\[
\phi(x) := (2q-q^2)x^2+2qx(1-x) = 2qx-q^2 x^2.
\]
It follows by elementary analysis that $\phi$ is strictly increasing and concave on $(0,p)$, with $\phi(0)=0$ and $\phi(p)=p$. Thus, we deduce by induction that $p_i(u)=p$ for every $i\geq 1$ and $u\in V$ as claimed. Similarly, for each $i \geq 1$ and adjacent $u,v \in V$ we have by definition of $\mu_{i+1}$ that 
\begin{align*}
\sigma_{i+1}(u,v)
&=
(1-(1-q)^2) \, \mu_i\bigl(\omega(u) = \omega(v)=1)^2
\\
&\hspace{3.5cm}
+ 2q \, \mu_i\bigl(\omega(u) = \omega(v) = 1\bigr) \,  (1-\mu_i\bigl(\omega(u) = \omega(v)=1))\\
&\hspace{5.8cm}+2q^2\mu_i\bigl(\omega(u) =1, \omega(v)=0\bigr) \, \mu_i\bigl(\omega(u) =0, \omega(v)=1\bigr)
\\ 
&= \phi(\sigma_i(u,v))+ 
2q^2\mu_i\bigl(\omega(u) =1, \omega(v)=0\bigr) \, \mu_i\bigl(\omega(u) =0, \omega(v)=1\bigr)\\ 
&\geq \phi(\sigma_i(u,v)),
\end{align*}
where we have used that fact that if $\omega(u)=\omega(v)=1$ then $u$ and $v$ are in the same cluster of $\omega$. 
Since $\phi$ is strictly increasing and concave on $(0,p)$, with the only fixed points $0$ and $p$, and since $\sigma_1(u,v)>0$, it follows that $\sigma_i(u,v) \uparrow p$ as $i\to \infty$. The claim now follows since
\begin{multline*}
\mu_i(\omega(u) \neq \omega(v)) = 
\mu_i(\omega(u) =1, \omega(v) =0)
+
\mu_i(\omega(u) =0, \omega(v) =1)
=
2p\left(1-\frac{\sigma_i(u,v)}{p}\right),
\end{multline*}
which tends to zero as $i\to \infty$.
\end{proof}

See \cref{fig:Z2,fig:Z3} for simulations of the measures $\mu_i$ on $\Z^2$ and $\Z^3$.

\subsection{Ergodicity and condensation}

On Cayley graphs of infinite Kazhdan groups, \cref{cor:degeneration} will be useful only if we also know something about the ergodicity of the measures $\mu_i$. To this end, we will apply some tools introduced by Lyons and Schramm \cite{LS99} that give sufficient conditions for ergodicity of $q$-thinned processes.
The first such lemma, which is proven in \cite[Lemma 4.2]{LS99} and is based on an argument of Burton and Keane \cite{burton1989density}, shows that every cluster of an invariant percolation process has an invariantly-defined  \emph{frequency} as measured by an independent random walk. Moreover, conditional on the percolation configuration, the frequency of each cluster is non-random and does not depend on the starting point of the random walk.

\begin{lemma}[Cluster frequencies]\label{l.freq}
Let $G=(V,E)$ be a Cayley graph of an infinite, finitely generated group $\Gamma$. There exists a Borel measurable, $\Gamma$-invariant function $\freq :\{0,1\}^V \to [0,1]$ with the following property. Let $\mu\in M(\Gamma,\Omega)$ be an  invariant site percolation, and let $\omega$ be a random variable with law $\mu$. Let $v$ be a vertex of $G$ and let $\P_v$ be the law of simple random walk $\{X_n\}_{n\ge 0}$ on $G$ started at $v$. Then
\begin{equation}\label{e.freq}
\lim_{N\to\infty} \frac{1}{N} \sum_{n=0}^{N-1} \mathbbm{1}_{\{X_n \in C\}} = \freq(C) \qquad \text{ for every cluster $C$ of $\omega$}
\end{equation}
$\mu \otimes \P_v$-almost surely. 
\end{lemma}


This notion of frequency is used in the next proposition, which is a slight variation on \cite[Lemma 6.4]{LS99}. We define $\sF\subseteq \{0,1\}^V$ to be the event that there exists a cluster of positive frequency. Note that the $\Gamma$-invariance and Borel measurability of $\freq$ implies that $\sF$ is $\Gamma$-invariant and Borel measurable also.

\begin{prop}[Ergodicity of the $q$-thinning]\label{p.qthin}
Let $G=(V,E)$ be a Cayley graph of an infinite, finitely generated group $\Gamma$, and let $\mu\in E(\Gamma,\Omega)$ be an ergodic invariant site percolation such that $\mu(\sF)=0$.  Then the $q$-thinned measure $\mu^q$ is also ergodic for every $q\in [0,1]$. Similarly, if we have $k$ measures $\nu_1,\ldots,\nu_k \in E(\Gamma,\Omega)$ such that $\nu_i(\sF)=0$ for every $1\leq i \leq k$ and  $\nu_1\otimes\dots\otimes\nu_k$ is ergodic, then $\nu_1^q\otimes\dots\otimes\nu_k^q$ is also ergodic for every $q\in [0,1]$.
\end{prop}

\begin{proof}
Let $\omega$ be a random variable with law $\mu \in M(\Gamma,\Omega)$. We first show that if $\mu(\sF)=0$ then 
\begin{equation}\label{e.twoballs}
\lim_{N\to\infty} \frac{1}{N} \sum_{n=0}^{N-1} \P_o\Big( B(X_0,r) \llra B(X_n,r) \Big) = 0\qquad \mu\text{\,-\,a.s.}, 
\end{equation}
for every $r\geq 0$, where $B(v,r)$ is the ball of radius $r$ around $v\in V$, and for $U_1,U_2\subseteq V$, we write $\{U_1 \llra U_2\}$ for the event that there exist $x_1\in U_1$ and $x_2\in U_2$ that are in the same cluster of $\omega$. An easy but important implication of~(\ref{e.twoballs}) is that 
\begin{equation}
\label{eq:twoballs2}
\inf_{x\in V} \mu\left( B(o,r) \llra B(x,r) \right) = 0
\end{equation}
for every $r \geq 0$ and every $\mu \in M(\Gamma,\Omega)$ such that $\mu(\sF)=0$. (Note that the proof of \cite[Lemma 6.4]{LS99} established this fact under the additional assumption that $\mu$ is \emph{insertion tolerant}.)

Condition on $\omega$, and denote the finitely many clusters that intersect $B(o,r)$ by $\{C_i\}_{i=1}^m$. Taking $\P_o$-expectations in~(\ref{e.freq}) and using the dominated convergence theorem, \cref{l.freq} implies that 
\begin{equation}\label{e.oneball}
\lim_{N\to\infty} \frac{1}{N} \sum_{n=0}^{N-1} \P_o\big( B(X_0,r) \llra X_n \big) = \lim_{N\to\infty} \frac{1}{N} \sum_{i=1}^m \sum_{n=0}^{N-1} \P_o\big( X_n \in C_i \big) = 
 0 \qquad \text{$\mu$-a.s.}
\end{equation}
Now notice that  \[\sum_{i=0}^r \P_o\Big( B(X_0,r) \llra X_{n+i} \Bmid B(X_0,r) \llra B(X_n,r) \Big) \geq \deg(o)^{-r}\] 
for every $n,r\geq 0$,  and hence that
\begin{equation}\label{e.degball}
\sum_{n=0}^{N-1} \P_o\big( B(X_0,r) \llra B(X_n,r) \big) \leq (r+1)\deg(o)^{r}\sum_{n=0}^{N-1+r} \P_o\big( B(X_0,r) \llra X_n \big)
\end{equation}
for every $N\geq 1$ and $r\geq 0$.
Dividing by $N$ and letting $N\to\infty$, this inequality and~(\ref{e.oneball}) imply~(\ref{e.twoballs}).

The rest of the proof of the ergodicity of $\mu^q$ is identical to the argument in \cite[Lemma 6.4]{LS99}, which we recall here for the reader's convenience. Suppose that $\mu$ is ergodic. Denote by $\omega^q$ the $q$-thinned configuration obtained from $\omega$, let $\P^q$ denote the joint law of $(\omega,\omega^q)$, and let $A$ be any invariant event for $(\omega,\omega^q)$. For every $\eps>0$ there exists some $r>0$ and an event $A_{\eps,r}$ depending only on the restriction of $(\omega,\omega^q)$ to $B(o,r)$ such that $\P^q\big(A \,\triangle\, A_{\eps,r}\big) < \eps$. By \eqref{eq:twoballs2} we may take $x$ such that  $\mu\big( B(o,r) \llra B(x,r) \big)<\eps$. Conditionally on $D_x:=\{B(o,r) \,\,\, \not\!\!\!\llra B(x,r)\}$ in $\omega$, the coin flips for the $q$-thinning of the clusters intersecting $B(o,r)$ and $B(x,r)$ are independent, hence
\begin{align*}
\Big| \P^q \big(A_{\eps,r} \cap \gamma_x A_{\eps,r} \bmid \omega \big) - \P^q \big(A_{\eps,r} \bmid \omega \big) \, \P^q \big(\gamma_x A_{\eps,r} \bmid \omega \big) \Big| \leq 2 \cdot \mathbf 1_{D_x}(\omega)\,,
\end{align*}
where $\gamma_x$ is translation by $x \in \Gamma$. Taking expectation with respect to $\mu$ then letting $\eps\to 0$, we get that 
$$
\E_\mu\Big| \P^q(A\mid \omega) -  \P^q(A\mid \omega)^2 \Big| = 0\,,
$$
and hence that $ \P^q (A\mid \omega) \in \{0,1\}$ $\mu$-almost surely. By the ergodicity of $\mu$, this implies that $\P^q(A) \in \{0,1\}$. It follows that $\P^q$ is ergodic and hence that $\mu^q$ is ergodic also. 

Similarly, if  $\nu_1\otimes\dots\otimes\nu_k$ is ergodic and $\nu_i(\sF)=0$ for every $1\leq i \leq k$, then we have by \eqref{e.twoballs} that if $\mathbf{\omega}=(\omega_1,\ldots,\omega_k)$ is a random variable with law  $\mathbf{\nu}=\nu_1\otimes\dots\otimes\nu_k$ then
\begin{multline*}
\inf_{x\in V} \nu\left( B(o,r) \leftrightarrow B(x,r) \text{ in $\omega_i$ for some $1\leq i \leq k$} \right) \\\leq \lim_{N\to\infty}  \frac{1}{N} \sum_{n=0}^{N-1} \sum_{i=1}^k \nu_i \otimes \P_o\Big( B(X_0,r) \llra B(X_n,r) \Big) =0.
\end{multline*}
 The ergodicity of $\nu_1^q\otimes\dots\otimes\nu_k^q$ then follows by a similar argument to that above.
\end{proof}

Define $i_{\mathrm{freq}}$ to be the minimal $i\geq 1$ such that $\mu_i(\sF)>0$, letting $i_{\mathrm{freq}}=\infty$ if this never occurs. We want to prove, using induction and \cref{p.qthin}, that $\mu_i$ is ergodic for every $1\leq i\leq i_{\mathrm{freq}}$. However, it is not always true that the union of two independent ergodic percolation processes is ergodic\footnote{Consider, for example, the random subset $\omega$ of $\Z$ in which $\omega(n)=\mathbbm{1}(n \text{ is odd})$ for every $n\in \Z$ with probability $1/2$ and otherwise  $\omega(n)=\mathbbm{1}(n \text{ is even})$ for every $n\in \Z$. The law of this process is shift invariant and ergodic, but the law of the union of two independent copies of the process is not ergodic.}. To circumvent this problem, we instead prove a slightly stronger statement. Recall that a measure $\mu\in M(\Gamma,\Omega)$ is \textbf{weakly mixing} 
if and only if 
the independent product $\mu \otimes \mu \in M(\Gamma,\Omega^2)$ is ergodic when $\Gamma$ acts diagonally on $\Omega^2$, if and only if 
the $k$-wise independent product $\mu^{\otimes k} \in M(\Gamma,\Omega^k)$ is ergodic for every $k\ge 2$ \cite[Theorem 1.24]{Walters}. This can be taken as the definition of weak mixing for the purposes of this paper. 

\begin{prop}
\label{prop:mixing}
Let $G$ be a Cayley graph of an infinite, finitely generated group $\Gamma$, let $p\in (0,1)$, and let $(\mu_i)_{i\geq 1}$ be as above. Then $\mu_i$ is weakly mixing for every $1\leq i\leq i_{\mathrm{freq}}$.
\end{prop}

\begin{proof} 
We will prove the claim by induction on $i$. 
For $i=1$, $\mu_1$ is simply the law of Bernoulli-$p$ percolation, which is certainly weakly mixing. Now assume that $i<i_{\mathrm{freq}}$ and that $\mu_i$ is weakly mixing, so that $\mu_i^{\otimes 4}$ is ergodic. 
Applying \cref{p.qthin} we obtain that the
independent 4-wise product $(\mu^q_i)^{\otimes 4}$ of the $q$-thinned percolations is again ergodic. Since $\mu_{i+1}^{\otimes 2}$ can be realized as a factor of $(\mu_i^q)^{\otimes 4}$ by taking the unions in the first and second halves of the 4 coordinates, and since factors of ergodic processes are ergodic, it follows that $\mu_{i+1}^{\otimes 2}$ is ergodic and hence that $\mu_{i+1}$ is weakly mixing.
\end{proof}

Since $\sF$ is an invariant event, \cref{prop:mixing} has the following immediate corollary.

\begin{corollary}
\label{cor:connectivity_at_iu}
Let $G$ be a Cayley graph of an infinite, finitely generated group $\Gamma$, let $p\in (0,1)$, and let $(\mu_i)_{i\geq 1}$ be as above. If $i_{\mathrm{freq}}<\infty$ then
$\mu_{i_{\mathrm{freq}}}(\sF)=1$. 
\end{corollary}

\begin{remark}\label{r.insuni}
It is possible to prove by induction that the measures $\mu_i$ are both \emph{insertion tolerant} and \emph{deletion tolerant} for every $i\geq 1$. Thus, it follows from the \emph{indistinguishability theorem} of Lyons and Schramm \cite{LS99}, which holds for all insertion tolerant invariant percolation processes, that if $i_{\mathrm{freq}} < \infty$ then $\mu_{i_{\mathrm{freq}}}$ is supported on configurations in which there is a unique infinite cluster; see \cite[Section 4]{LS99}. We will not require this result.
\end{remark}

Next, we deduce the following from \cref{prop:mixing}.

\begin{corollary}[Condensation]
\label{prop:termination}
Let $G$ be a Cayley graph of a countably infinite Kazhdan group, let $p\in (0,1)$ and let $(\mu_i)_{i\geq 1}$ be as above. Then $i_{\mathrm{freq}}<\infty$.
\end{corollary}

\begin{proof}
Suppose for contradiction that $i_{\mathrm{freq}}=\infty$. Then it follows by \cref{prop:mixing} that $\mu_i$ is weakly mixing and hence ergodic for every $i\geq 1$. But $\mu_i$ weak$^*$ converges to the non-ergodic measure $p\delta_V +(1-p)\delta_\emptyset$ by \cref{cor:degeneration}, contradicting property (T). 
\end{proof}

\begin{proof}[Proof of \cref{thm:main}]
Recall that every countable Kazhdan group is finitely generated \cite[Theorem 1.3.1]{MR2415834}. 
Let $G=(V,E)$ be a Cayley graph of $\Gamma$, let $p\in (0,1)$, and let $(\mu_i)_{i\geq 1}$ be as above. It follows from \cref{prop:termination,cor:connectivity_at_iu} that $1\leq i_{\mathrm{freq}}<\infty$ and that $\mu_{i_{\mathrm{freq}}}$ is supported on $\sF$. Let $\omega \in \{0,1\}^V$ be sampled from $\mu_{i_{\mathrm{freq}}}$, so that $\omega \in \sF$ almost surely. 
Fatou's lemma implies that the total frequency of all components of $\omega$ is at most $1$ almost surely, and consequently that $\omega$ has at most finitely many components of maximal frequency almost surely. 
Let $\omega'$ be obtained from $\omega$ by choosing one of the maximum-frequency components of $\omega$ uniformly at random, retaining this component, and deleting all other components of $\omega$, so that  $\omega'$ has a unique infinite cluster almost surely. Let $\eta\in \{0,1\}^{\Gamma \times \Gamma}$ be defined by setting $\eta(u,v)=1$ if and only if $u$ and $v$ are adjacent in $G$ and have $\omega'(u)=\omega'(v)=1$, and let $\nu$ be the law of $\eta$, so that $\nu \in M(\Gamma,\conn(\Gamma))$. It follows by \cref{prop:sparse_cost,cor:degeneration} that
\[
\cost(\Gamma) \leq 1 + \frac{1}{2}\int_{\conn(\Gamma)} \deg_\eta(o) \dif \nu(\eta)
\leq  
1 + \frac{\deg(o)}{2}\int_{\Omega} \omega(o) \dif \mu_{i_{\mathrm{freq}}}(\omega)
= 1 + \frac{p \deg(o)}{2}.
\]
The claim now follows since $p\in(0,1)$ was arbitrary.
\end{proof}

\section{Relative property (T)}\label{s.rel}

In this section we sketch an extension of our results to groups with \emph{relative property $(T)$}, a notion that was considered implicitly in the original work of Kazhdan \cite{Kazhdan} and first studied explicitly by Margulis \cite{MR939574}.
If $H$ is a subgroup of $\Gamma$, then the pair $(\Gamma,H)$ is said to have {\bf relative property (T)} if every unitary representation of $\Gamma$ on a Hilbert space that has almost-invariant vectors has a non-zero $H$-invariant vector; see \cite[Definition 1.4.3]{MR2415834}. For example, $(\Z^2 \rtimes \mathrm{SL}_2(\Z),\Z^2)$ has relative property (T) but $\Z^2 \rtimes \mathrm{SL}_2(\Z)$ does not have property (T) itself \cite{Kazhdan}. Similar results with $\Z$ replaced by other rings have been proven by Kassabov \cite{MR2342638} and Shalom \cite{MR1813225}. 
See e.g.\ \cite{jaudon2007notes,de2006relative} for further background.

The analogue of the Glasner-Weiss theorem for pairs $(\Gamma,H)$ with relative property $(T)$ is that any  weak$^*$-limit of $\Gamma$-invariant $H$-ergodic probability measures on $\Omega=\{0,1\}^\Gamma$ is $\Gamma$-ergodic; this can be established using the same methods as those of \cite{MR1475550}. 
Using this, our proof of \cref{thm:main} can be extended to the following situation:

\begin{thm}
\label{thm:rel}
Let $H$ be an infinite \emph{normal} subgroup of a countable group $\Gamma$, and assume that the pair $(\Gamma,H)$ has {relative property (T)}.
Then $\Gamma$ has cost $1$.
\end{thm}

The fact that $\Gamma$ has $\beta_1(\Gamma)=0$ under the hypotheses of \cref{thm:rel} was proven by Martin \cite{martin2003analyse}. 
The assumption that $H$ is infinite is clearly needed since every group has relative property $(T)$ with respect to its one-element subgroup. It should however be possible to relax the condition of normality in various ways, for example to \emph{$s$-normality} \cite{MR2827095} or \emph{weak quasi-normality} \cite{MR2777775}. We do not pursue this here.

It is a theorem of Gaboriau \cite[Theorem 3.4]{MR1919400} that if $\Gamma$ is a countable group with an infinite, infinite-index, normal subgroup\footnote{In \cite{MR1919400}, it is assumed that $H$ is finitely generated, but it is well-known that this can be replaced by the weaker assumption that $H$ has finite cost, by a co-induction argument similar to the one in our proof below.} $H$ with $\cost(H)<\infty$, then $\Gamma$ has cost $1$. The condition $\cost(H)<\infty$ is very weak, and applies in particular whenever $H$ is either finitely generated or amenable. Thus, most natural examples to which \cref{thm:rel} applies are already treated either by this theorem or by \cref{thm:main} (in the case $H=\Gamma$). As such, the main interest of \cref{thm:rel} is to demonstrate the flexibility of the proof of \cref{thm:main}, and we give only a brief sketch of the proof.


\begin{proof}[Sketch of proof] First assume that $\Gamma$ is finitely generated. We start with the same sequence of measures $\{\mu_i\}_{i\ge 1}$ on $\Omega$ as before, using a Cayley graph $G$ of $\Gamma$ with a finite symmetric generating set $S$, with edges given by right multiplication by the generating elements. The left cosets $gH$ then form a partition of the Cayley graph into isomorphic subgraphs. Moreover, if two cosets $g_1H$ and $g_2H$ are neighbours in the sense that $g_1n_1s=g_2n_2$ for some $n_i\in H$ and $s\in S$, then for every $n\in H$ we have that $$g_1ns=g_1n n_1^{-1}g_1^{-1}g_2n_2=n' g_2 n_2=g_2 n''$$
for some $n',n''\in H$, because $H$ is normal. Thus, neighbouring cosets are connected in $G$ by infinitely many edges (because $H$ is infinite). 

We will have to measure cluster frequencies inside individual $H$-cosets, and will therefore use a random walk whose jump distribution generates $H$. Specifically, we enumerate the elements of $H$ as $\{h_1,h_2,\ldots\}$, let $(Z_i)_{i\geq 1}$ be an i.i.d.\ sequence of $H$-valued random variables with $\P(Z_i=h_k) = 2^{-k}$, and write $\P^{X_0}$ for the law of the random walk $(X_n)_{n\geq 0}$  defined by $X_{n+1} = X_n Z_{n+1}$ for each $n\geq 0$, where $X_0$ is an arbitrary element of $\Gamma$.
 An analogue of \cref{l.freq} is that for  every $r\in\N$, and every left $H$-coset $gH$, there exists an $H$-invariant cluster frequency function $\freq_{gH,r}$ such that if $\mu \in M(\Gamma,\Omega)$ is a $\Gamma$-invariant percolation process and $\omega$ is a random variable with law $\mu$ then
$$
\lim_{N\to\infty} \frac{1}{N} \sum_{n=0}^{N-1} \mathbbm{1}_{\big\{B(X_n,r) \cap C \not= \emptyset\big\}} = \freq_{gH,r}(C) \qquad \text{ for every cluster $C$ of $\omega$}
$$
$\mu \otimes \P^{X_0}$ almost surely for every $X_0 \in g H$.
The argument of \cref{p.qthin} then implies  that, if all cluster frequencies $\freq_{gH,r}(C)$ for every $r\in\N$ are almost surely zero in an $H$-invariant $H$-ergodic percolation measure $\mu$, then $\mu^q\otimes\dots\otimes\mu^q$ is $H$-ergodic. The reason we need the zero frequencies for all $r$-balls instead of just $r=0$ is that (\ref{e.degball}) does not necessarily hold now, since the random walk is confined to the $H$-coset, while percolation clusters are not.

Now, the analogue of \cref{prop:termination} is that if $(\Gamma,H)$ has relative property (T), then there exists $i_{\mathrm{freq}}<\infty$ and $r \in \N$ such that, if $\omega \in \{0,1\}^\Gamma$ is a random variable with law $\mu_{i_\mathrm{freq}}$, then for every left $H$-coset $gH$ there almost surely exists a cluster $C_{gH}$ with $\freq_{gH,r}(C_{gH})>0$. For each coset $gH$, let $\eta_{gH}$ be a cluster chosen uniformly at random from among those maximizing $\freq_{gH,r}$. 
%
 Now we can apply {\it sprinkling}: for any $\eps>0$, adding an independent Bernoulli$(\eps)$ bond percolation will almost surely connect the infinite clusters $\eta_{gH}$ in neighbouring $H$-cosets, and by deleting all clusters of the resulting percolation configuration other than the unique cluster containing $\bigcup \eta_{gH}$, we obtain a $\Gamma$-invariant percolation process of average degree $O(p+\eps)$ that has a unique infinite cluster. The fact that this sprinkling achieves the desired effect follows by a standard argument in invariant percolation (see e.g.\ the proof of \cite[Theorem 6.12]{LS99}), sketched as follows: 
\begin{enumerate}
\item 
Let $e$ be the identity element of $\Gamma$. For each $\delta>0$ there exists $R$ such that the cluster $\eta_{H}$ intersects the ball $B(e,R)$ with probability at least $1-\delta$. Thus, for each $u \in H$ and $s\in S$ the clusters $\eta_{sH}$ and $\eta_{H}$ both intersect the ball $B(u,R+1)$ with probability at least $1-2\delta$. Thus, if $u_1,u_2\ldots$ is an enumeration of $H$ then the clusters $\eta_{sH}$ and $\eta_{H}$ both intersect the ball $B(u_i,R+1)$ for infinitely many $i$ with probability at least $1-2\delta$ by Fatou's lemma.
On this event, it is immediate that the $\eps$-sprinkling connects the clusters $\eta_{sH}$ and $\eta_{H}$ almost surely. We deduce that  the $\eps$-sprinkling connects the clusters $\eta_{sH}$ and $\eta_{H}$  with probability at least $1-2\delta$, and hence with probability $1$ since $\delta>0$ was arbitrary.
\item Any two cosets have a finite chain of neighbouring coset pairs connecting them, hence sprinkling gives a \emph{unique} infinite cluster that contains $\bigcup \eta_{gH}$. 
\end{enumerate}
Since $p$ and $\eps$ can be made arbitrarily small, \cref{prop:sparse_cost} applies, and $\Gamma$ must have cost 1.

We can now remove the assumption that $\Gamma$ be finitely generated, as pointed out to us by Damien Gaboriau. First, the standard proof that Kazhdan groups are finitely generated \cite[Theorem 1.3.1]{MR2415834} gives for relative property (T) that the subgroup $H$ is contained in a finitely generated subgroup $\Gamma'$ of $\Gamma$ such that the pair $(\Gamma',H)$ has relative property (T) 
\cite[Theorems 2.2.1 and 2.2.3]{jaudon2007notes}. 
Our above proof gives that $\Gamma'$ has cost 1.  
Thus, for any $\eps>0$, we can independently take a $\Gamma'$-invariant random graph spanning $g\Gamma'$ with expected degree at most $2+\eps$ in each left coset $g\Gamma'$ of $\Gamma$. The resulting bond percolation $\omega_\eps$ is $\Gamma$-invariant. (This is the probabilistic interpretation of lifting the $\Gamma'$-action to a $\Gamma$-action by {\it co-induction}, as defined in \cite[Section 3.4]{GabExamples} or \cite[Section 10.(G)]{MR2583950}.)  
Let $\{\gamma_i : i \geq 1\}$ be an enumeration of $\Gamma$, and consider the random subset $\eta_\eps \subseteq \Gamma \times \Gamma$ in which each $(g,g\gamma_i)$ is included independently at random with probability $\eps 2^{-i}$. Let $\bar \eta_\eps = \eta_\eps \cup \{(g_1,g_2) : (g_2,g_1) \in \eta_\eps\}$ be obtained from $\eta_\eps$ by symmetrization, so that $\bar \eta_\eps$ is a $\Gamma$-invariant random graph on $\Gamma$ with expected degree at most $2\eps$. 
   Consider the independent union of $\omega_\eps$ and $\bar \eta_\eps$, which has expected degree at most $2+3\eps$. Since $H$ is an infinite normal subgroup of $\Gamma$ and each left $H$-coset $gH$ is contained in a single connected component of $\omega_\eps$, a similar argument to above shows that $\bar \eta_\eps$ almost surely connects each pair of components of $\omega_\eps$, so that the union of $\omega_\eps$ and $\bar \eta_\eps$ is connected almost surely. Since $\eps$ was arbitrary, $\Gamma$ has cost 1.
\end{proof}

\section{Closing remarks}


\begin{remark}
It would be interesting to investigate the behaviour of the processes we construct in \cref{ss.constru} on other classes of Cayley graphs. Simulations suggest, perhaps surprisingly, that the process has very different behaviours on $\Z^2$ and $\Z^3$: It seems that in two dimensions, when $p>0$ is small, $\mu_i$ is supported on configurations with no infinite clusters for every $i\geq 1$, while in three dimensions there is a unique infinite cluster after finitely many iterations. See \cref{fig:Z2} and \cref{fig:Z3}. Understanding the reason for this disparity may lead to proofs of cost $1$ for other classes of groups.
\end{remark}

\begin{figure}
    \centering
    \begin{subfigure}[b]{0.31\textwidth}
        \setlength{\fboxrule}{0.5pt}
\setlength{\fboxsep}{0pt}
\fbox{\includegraphics[trim = {0.1cm 0.1cm 0.1cm 0.1cm}, clip, width=\textwidth]{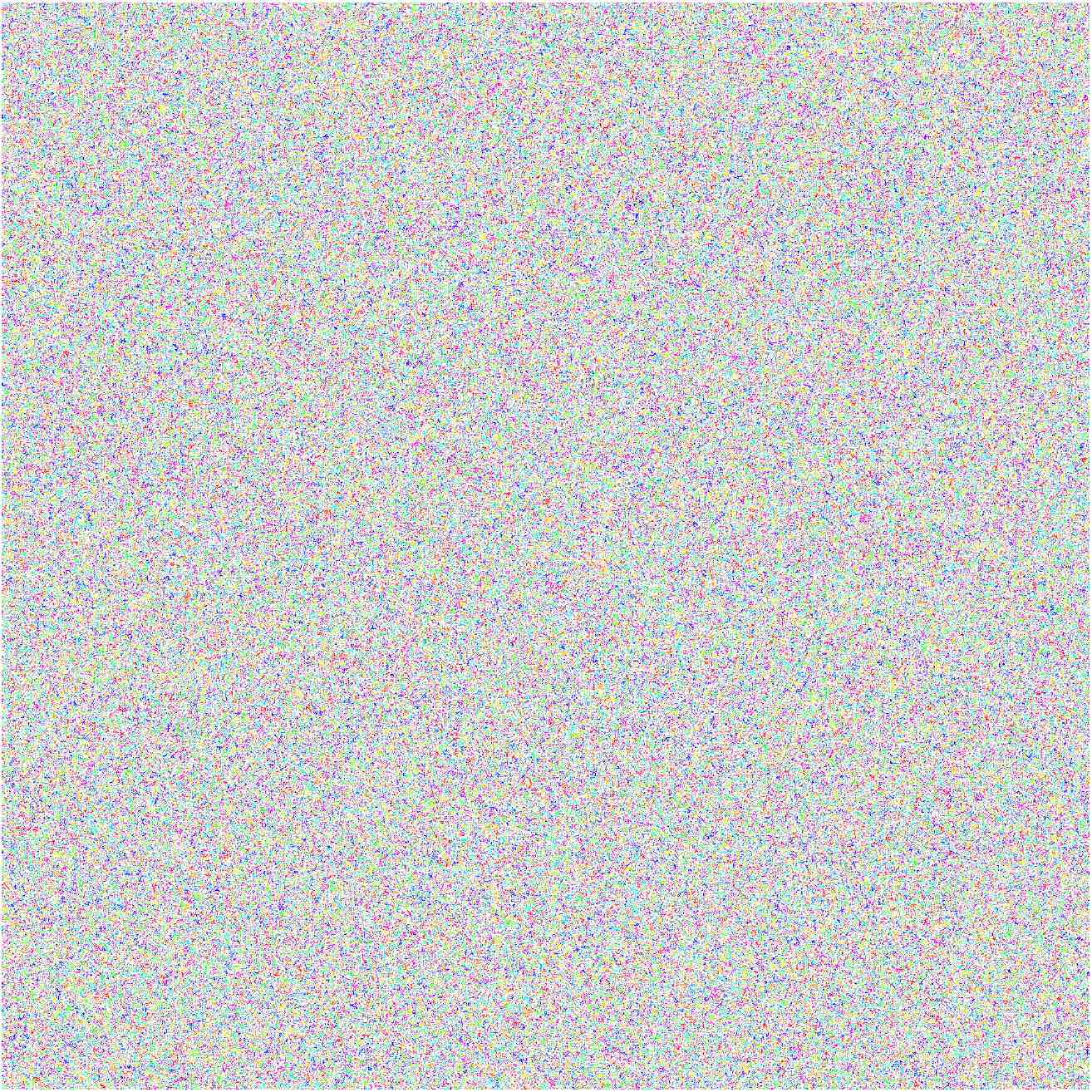}}
        \caption{Bernoulli percolation: $\mu_1$}
    \end{subfigure}
    \vspace{0.01\textwidth}
    \hspace{0.02\textwidth}
        \begin{subfigure}[b]{0.31\textwidth}
            \setlength{\fboxrule}{0.5pt}
\setlength{\fboxsep}{0pt}
\fbox{\includegraphics[trim = {0.1cm 0.1cm 0.1cm 0.1cm}, clip, width=\textwidth]{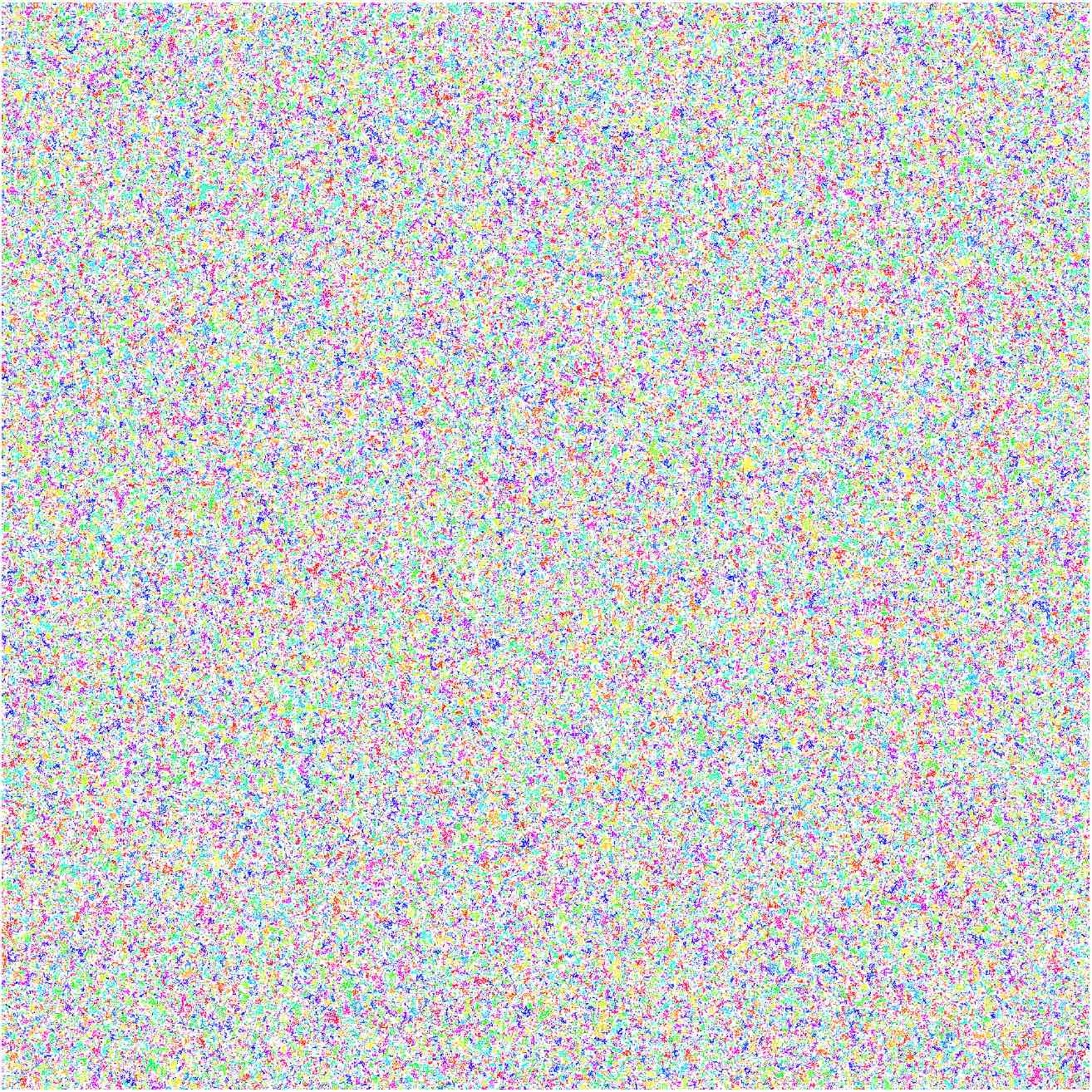}}
        \caption{One iteration: $\mu_2$}
    \end{subfigure}
\hspace{0.02\textwidth}
    \begin{subfigure}[b]{0.31\textwidth}
            \setlength{\fboxrule}{0.5pt}
\setlength{\fboxsep}{0pt}
\fbox{\includegraphics[trim = {0.1cm 0.1cm 0.1cm 0.1cm}, clip, width=\textwidth]{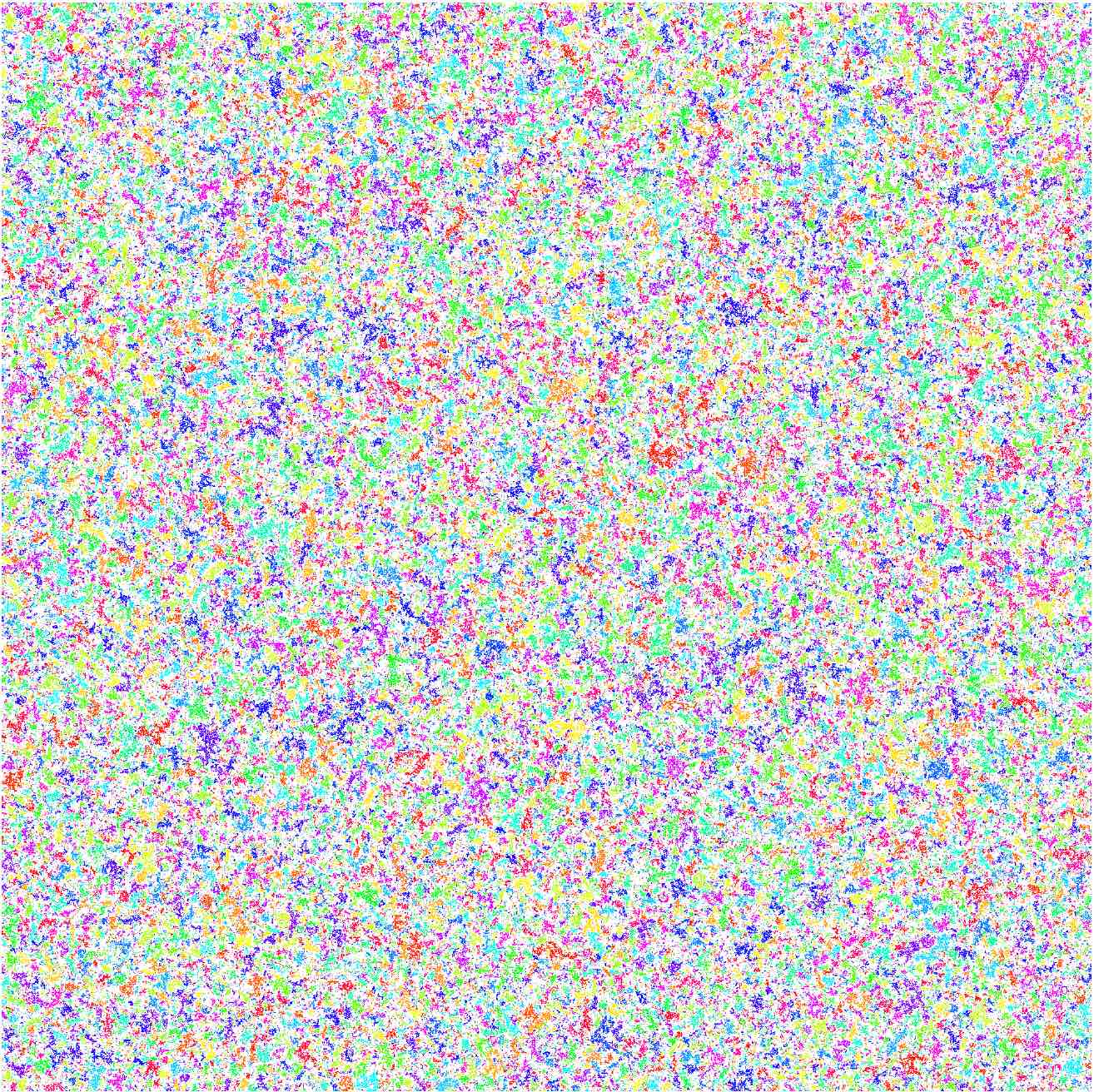}}
        \caption{Two iterations: $\mu_3$}
    \end{subfigure}
        \vspace{0.01\textwidth}
        \begin{subfigure}[b]{0.31\textwidth}
            \setlength{\fboxrule}{0.5pt}
\setlength{\fboxsep}{0pt}
\fbox{\includegraphics[trim = {0.1cm 0.1cm 0.1cm 0.1cm}, clip, width=\textwidth]{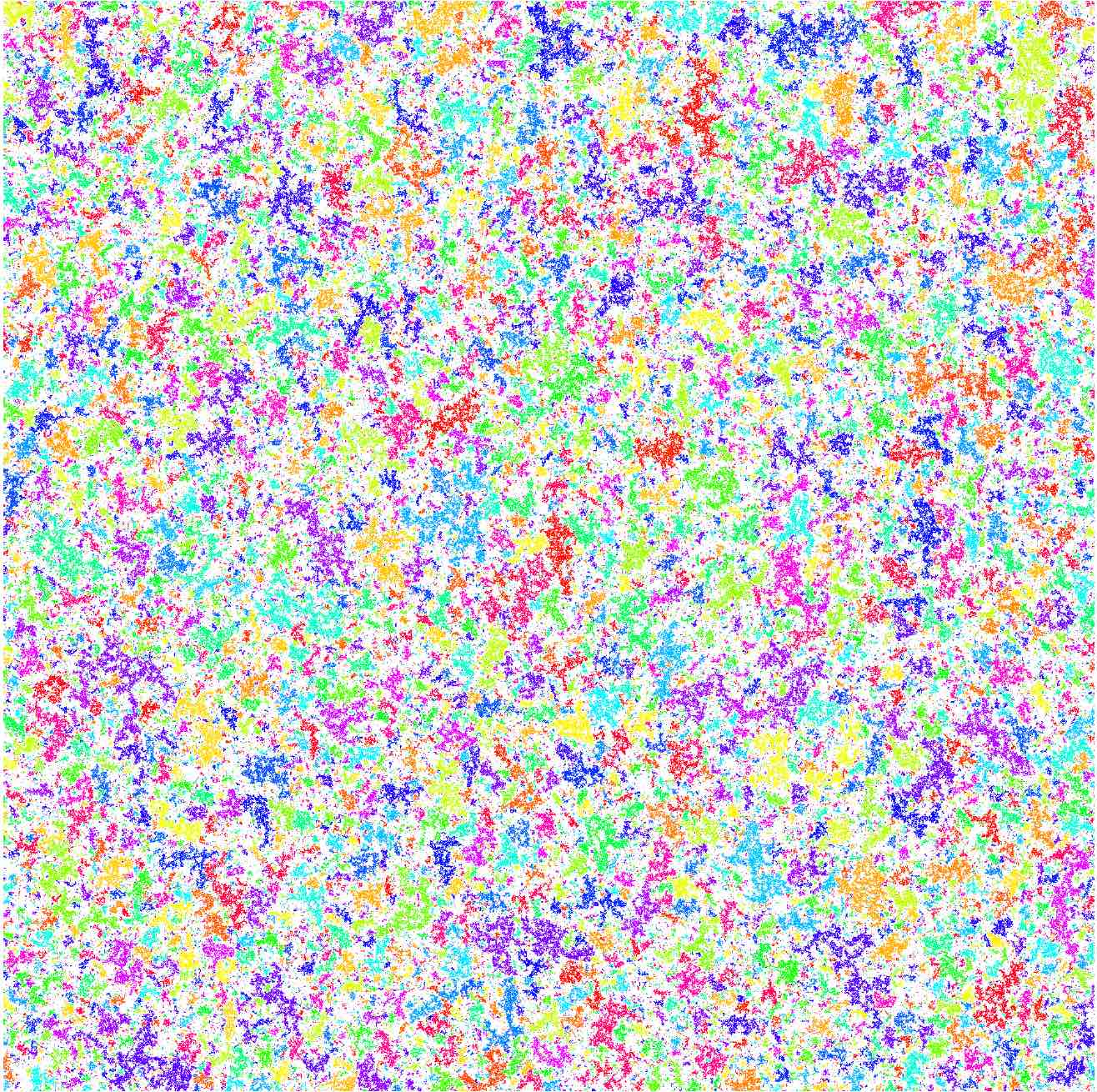}}
        \caption{Three iterations: $\mu_4$}
    \end{subfigure}
    \hspace{0.02\textwidth}
        \begin{subfigure}[b]{0.31\textwidth}
            \setlength{\fboxrule}{0.5pt}
\setlength{\fboxsep}{0pt}
\fbox{\includegraphics[trim = {0.1cm 0.1cm 0.1cm 0.1cm}, clip, width=\textwidth]{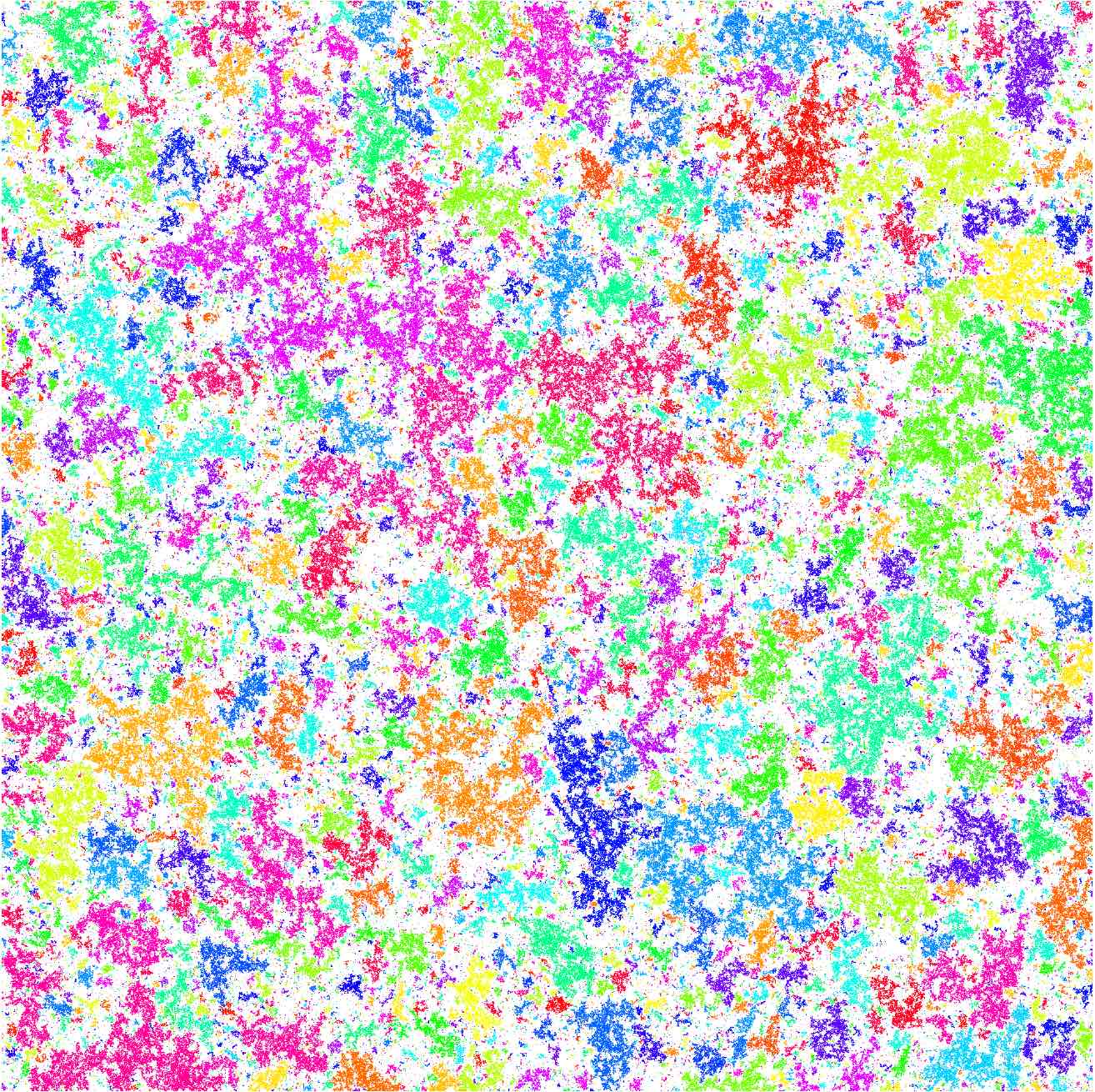}}
        \caption{Four iterations: $\mu_5$}
    \end{subfigure}
    \hspace{0.02\textwidth}
        \begin{subfigure}[b]{0.31\textwidth}
            \setlength{\fboxrule}{0.5pt}
\setlength{\fboxsep}{0pt}
\fbox{\includegraphics[trim = {0.1cm 0.1cm 0.1cm 0.1cm}, clip, width=\textwidth]{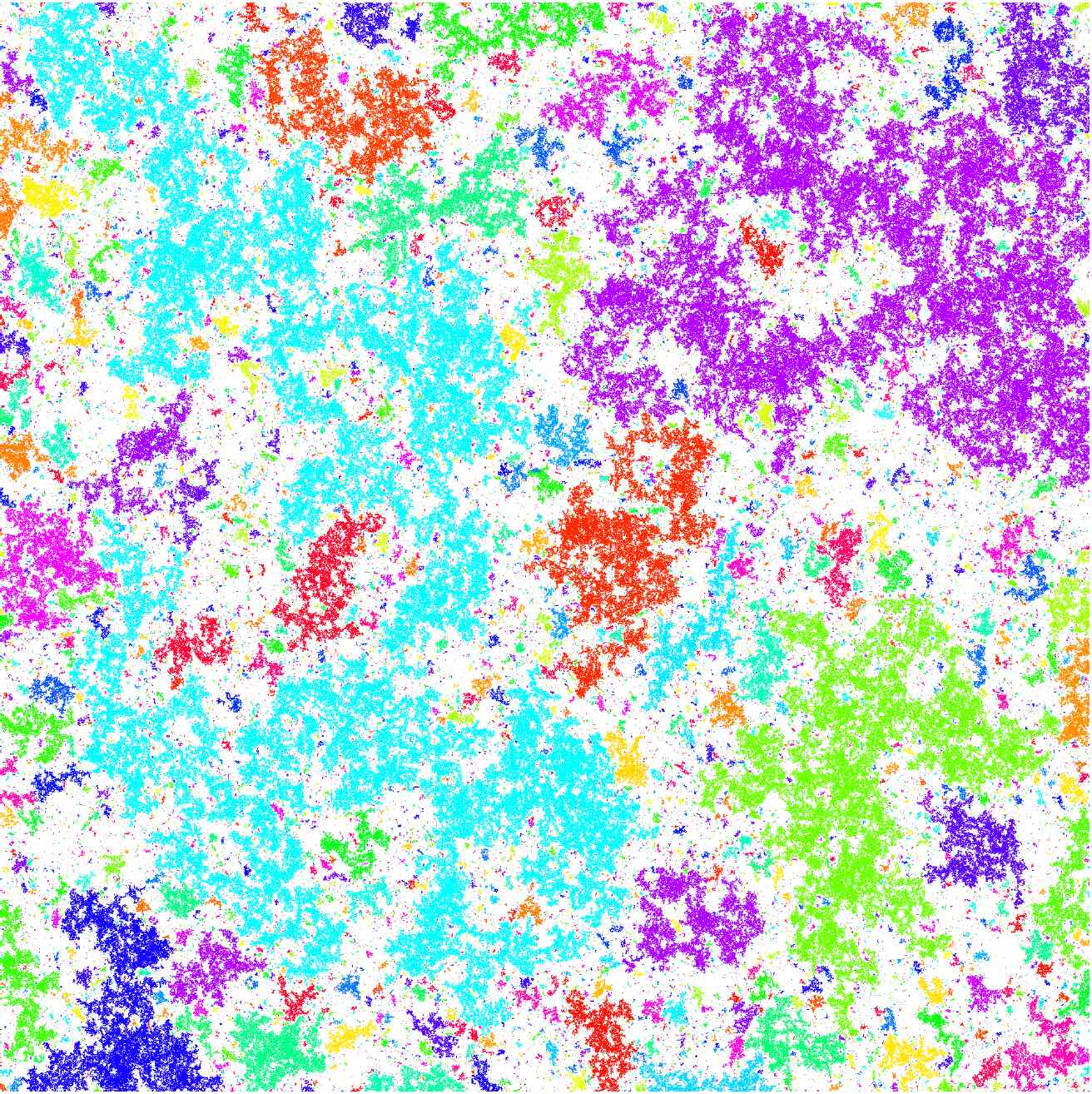}}
        \caption{Five iterations: $\mu_6$}
    \end{subfigure}
        \begin{subfigure}[b]{0.31\textwidth}
            \setlength{\fboxrule}{0.5pt}
\setlength{\fboxsep}{0pt}
\fbox{\includegraphics[trim = {0.1cm 0.1cm 0.1cm 0.1cm}, clip, width=\textwidth]{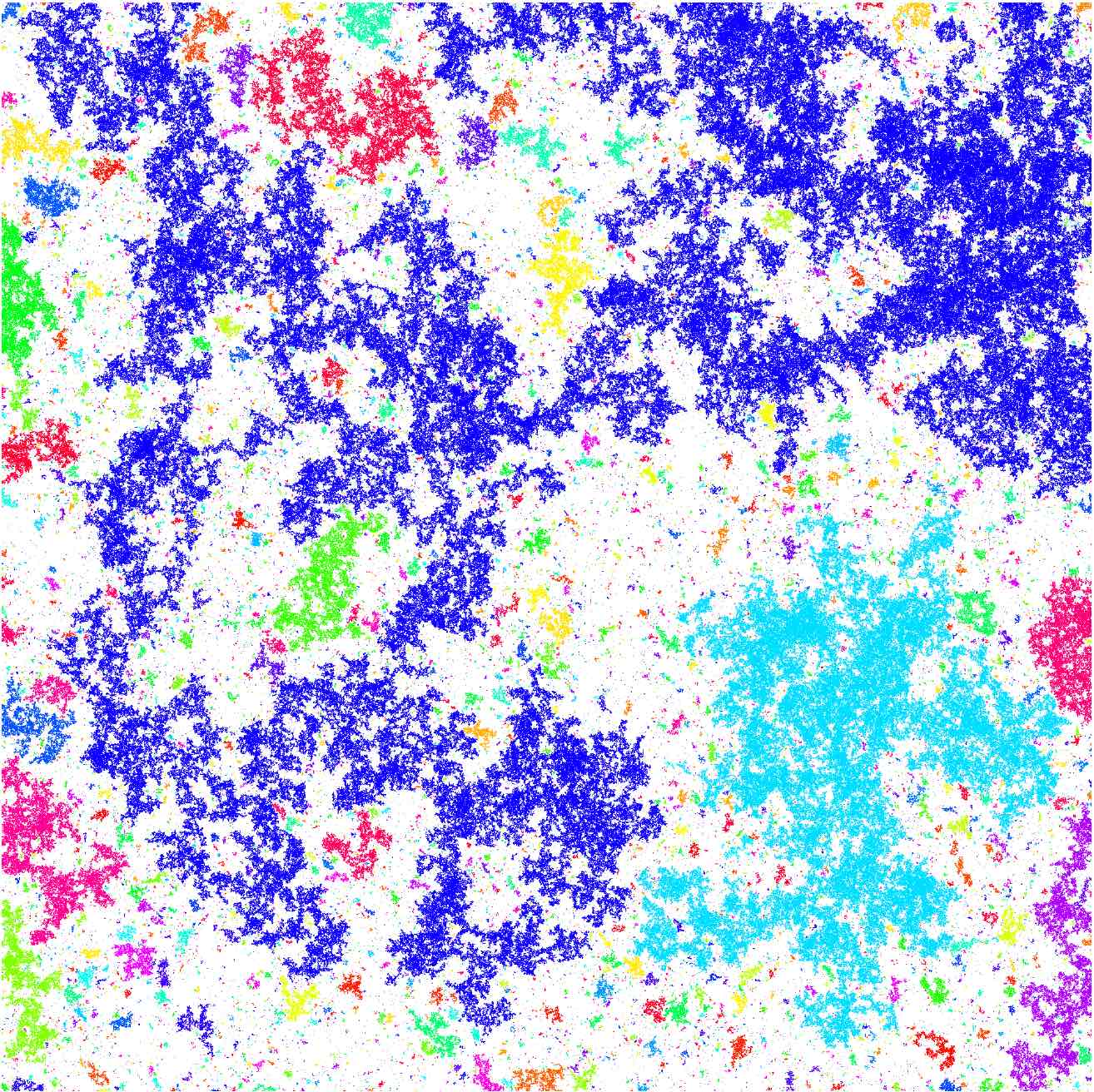}}
        \caption{Six iterations: $\mu_7$}
    \end{subfigure}
    \hspace{0.02\textwidth}
        \begin{subfigure}[b]{0.31\textwidth}
            \setlength{\fboxrule}{0.5pt}
\setlength{\fboxsep}{0pt}
\fbox{\includegraphics[trim = {0.1cm 0.1cm 0.1cm 0.1cm}, clip, width=\textwidth]{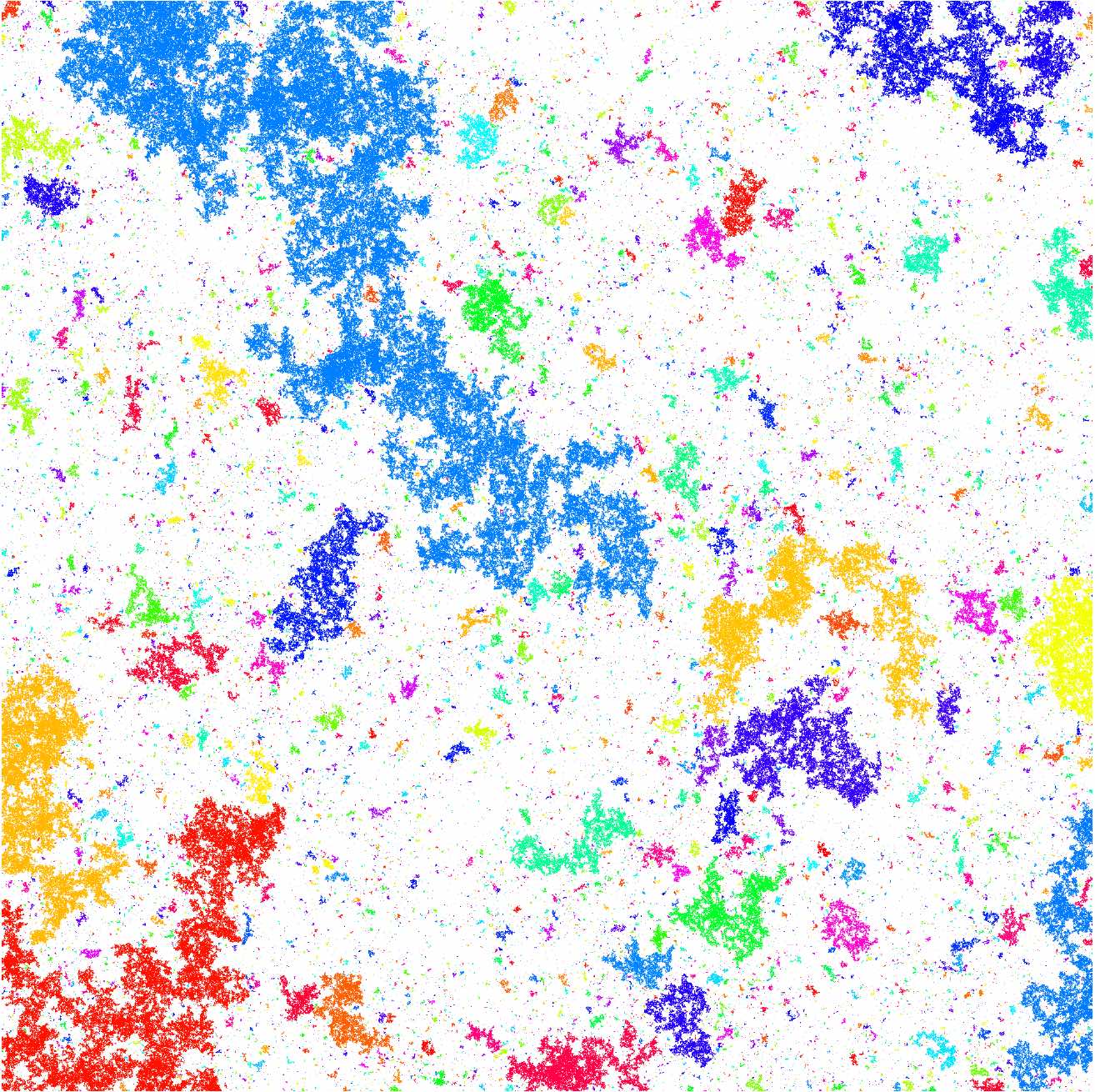}}
        \caption{Seven iterations: $\mu_8$}
    \end{subfigure}
    \hspace{0.02\textwidth}
        \begin{subfigure}[b]{0.31\textwidth}
            \setlength{\fboxrule}{0.5pt}
\setlength{\fboxsep}{0pt}
\fbox{\includegraphics[trim = {0.1cm 0.1cm 0.1cm 0.1cm}, clip, width=\textwidth]{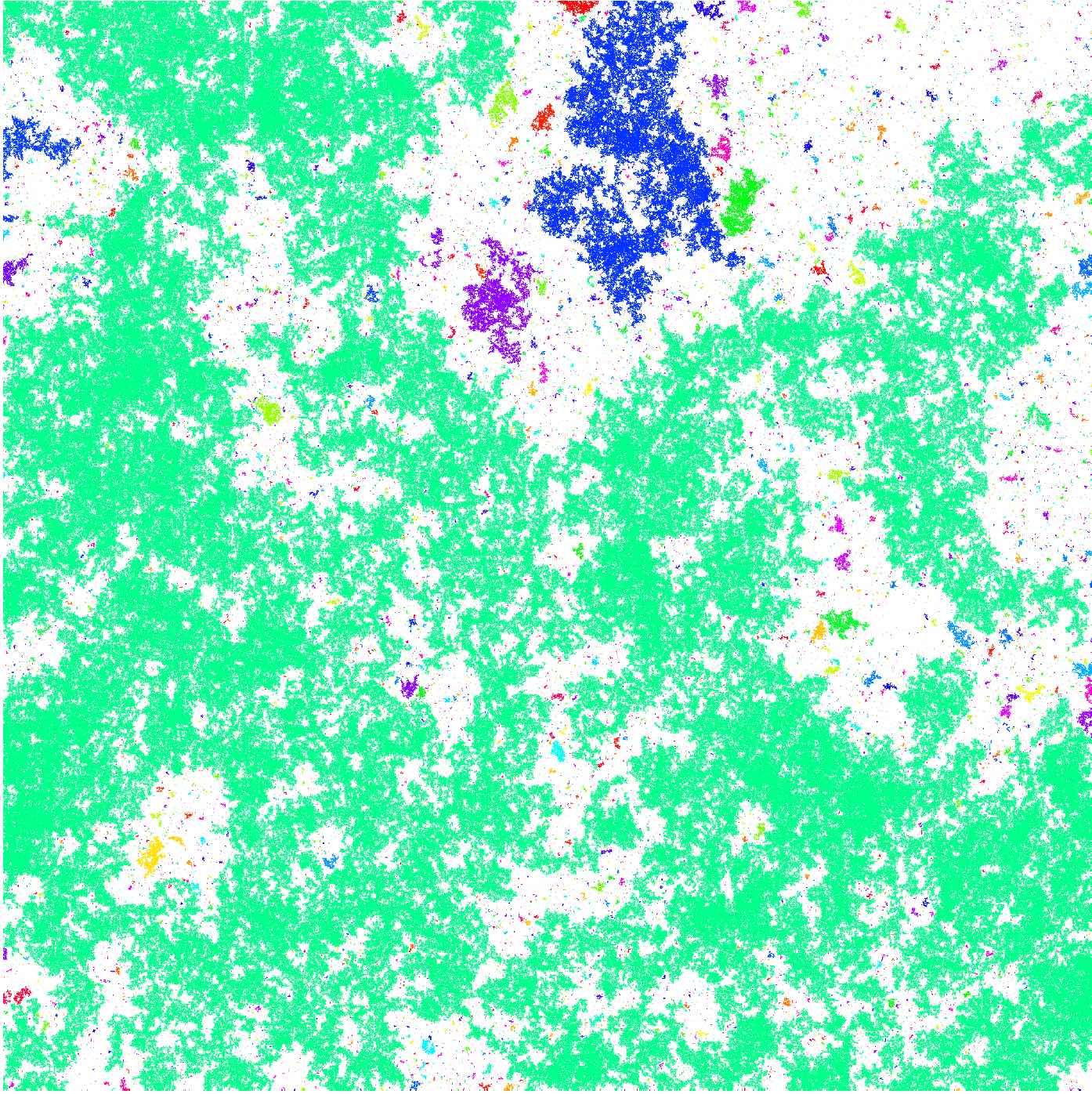}}
        \caption{Eight iterations: $\mu_9$}
    \end{subfigure}
    \caption{Simulations of the percolation processes constructed by iterative `thininning and independent union' as in \cref{ss.constru} with $p=0.35$ on a $2000$ by $2000$ box in the square lattice $\Z^2$. Unoccupied squares are white, while each cluster of occupied squares has been given a random colour for visualization purposes. In each case, the displayed configuration sampled from $\mu_{i+1}$ was obtained by taking the displayed configuration sampled from $\mu_{i}$ together with another independent configuration sampled from $\mu_i$, and then performing the procedure described in \cref{ss.constru}. These simulations strongly suggest that, for small densities on two-dimensional lattices, $\mu_i$ is supported on configurations with no infinite clusters for every $i\geq 1$. Note that the large clusters appear to have an interesting fractal-like structure similar to that which appears in critical percolation models. 
    }
    \label{fig:Z2}
\end{figure}

\begin{figure}
\centering
    \begin{subfigure}[b]{0.31\textwidth}
        \setlength{\fboxrule}{0.5pt}
\setlength{\fboxsep}{0pt}
\fbox{\includegraphics[trim = {18.1cm 3.3cm 16.5cm 3.36cm}, clip, width=\textwidth]{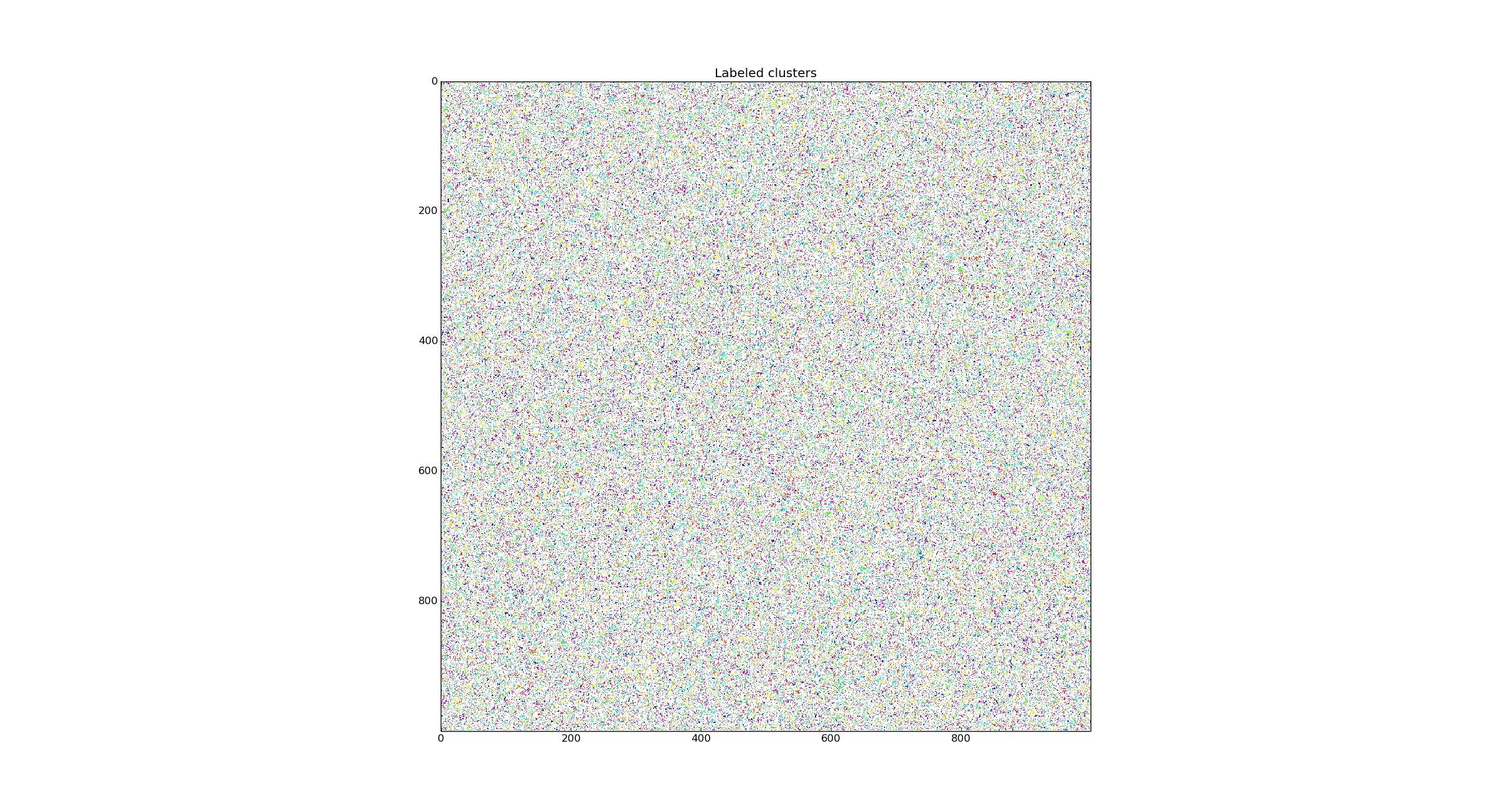}}
        \caption{Bernoulli percolation: $\mu_1$}
    \end{subfigure}
    \hspace{0.02\textwidth}
    \begin{subfigure}[b]{0.31\textwidth}
        \setlength{\fboxrule}{0.5pt}
\setlength{\fboxsep}{0pt}
\fbox{\includegraphics[trim = {18cm 3.3cm 16.5cm 3.36cm}, clip, width=\textwidth]{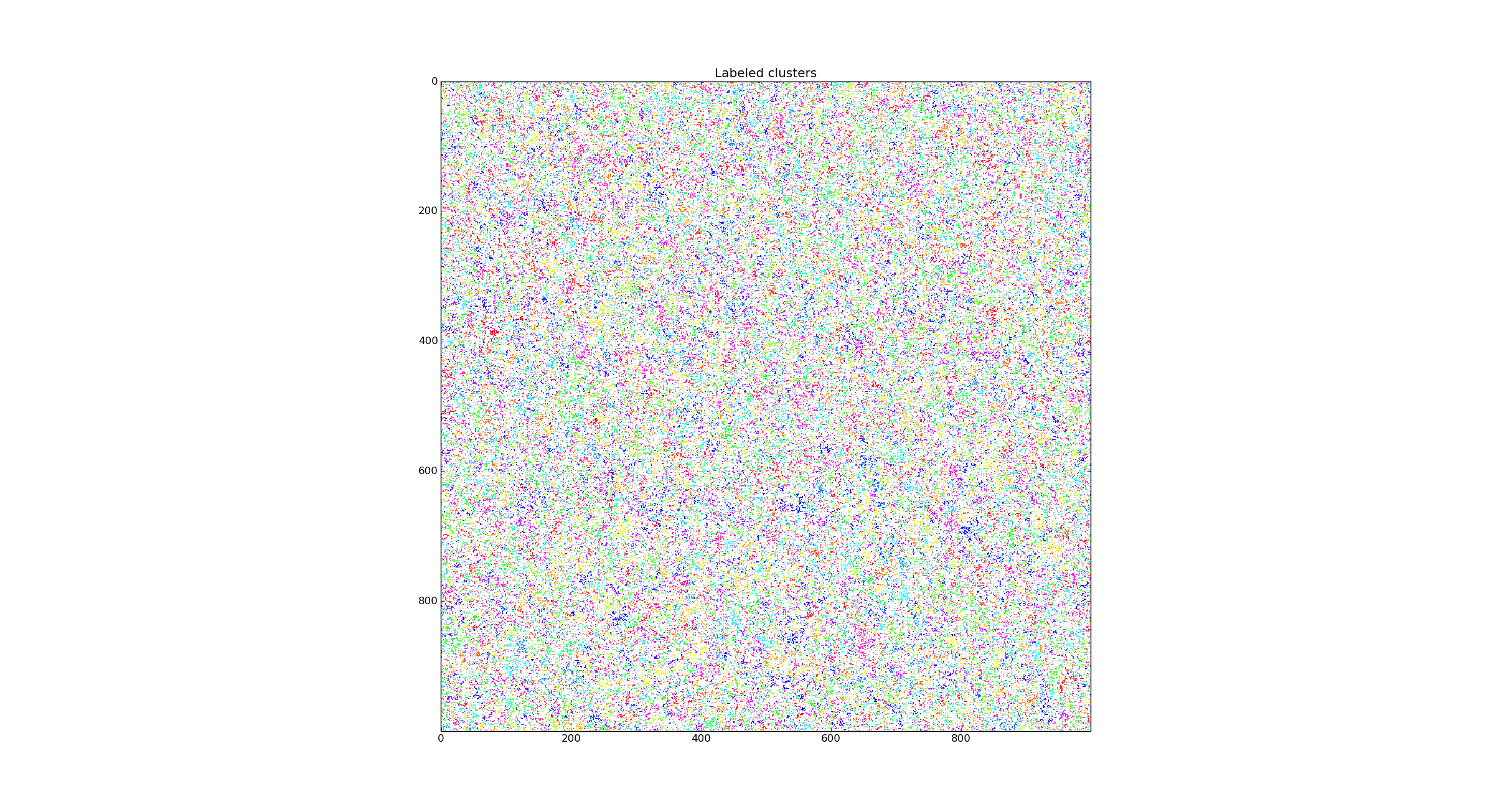}}
        \caption{One iteration: $\mu_2$}
    \end{subfigure}
    \hspace{0.02\textwidth}
        \begin{subfigure}[b]{0.31\textwidth}
        \setlength{\fboxrule}{0.5pt}
\setlength{\fboxsep}{0pt}
\fbox{\includegraphics[trim = {0.04cm 0.04cm 0.04cm 0.04cm}, clip, width=\textwidth]{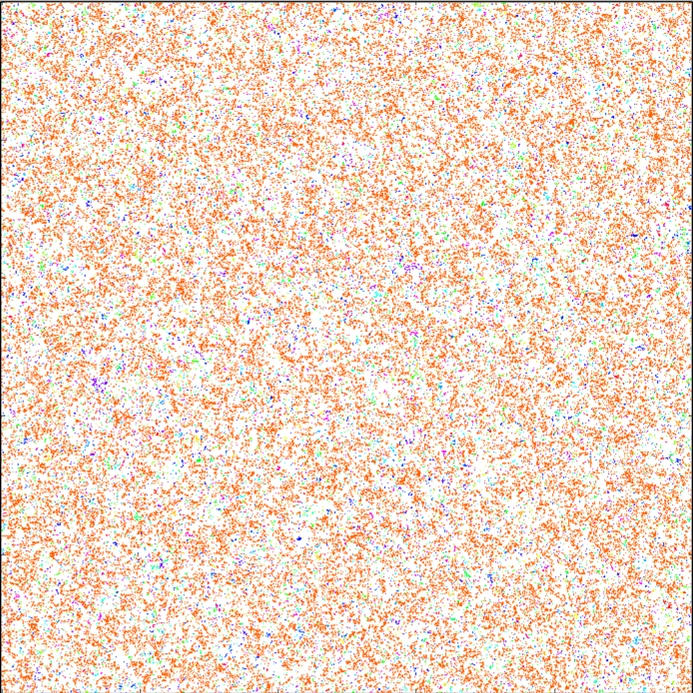}}
        \caption{Two iterations: $\mu_3$}
    \end{subfigure}
    \caption{Simulations of the percolation processes constructed by iterative `thininning and independent union' as in \cref{ss.constru} with $p=0.23$ on a $1000$ by $1000$ by $1000$ box in the cubic lattice $\Z^3$. The simulation sampled in each figure is independent of those in the other figures. Here we have sampled the process on the whole box, computed the clusters, and have presented the random equivalence relation that these clusters induce on the two-dimensional slice $[1,1000]^2 \times \{500\}$. Unoccupied cubes are white, while each slice of a 3d cluster of occupied cubes has been given a random colour for visualization purposes. In contrast to the two-dimensional case, but similarly to our primary setting of Kazhdan groups, it appears that a unique infinite (green) cluster emerges after two iterations. 
    }
    \label{fig:Z3}
\end{figure}

\begin{remark}\label{r.insertion}
Instead of relying on \cref{p.qthin} and working with cluster frequencies directly, one could instead write down a proof of the insertion tolerance of our measures $\mu_i$ (which is true though not completely immediate), then use \cite[Theorem 4.1 and Lemma 6.4]{LS99} of Lyons and Schramm almost as a black box. See also \cref{r.insuni}.
\end{remark}

\begin{remark}
Reflecting on the proof of \cref{thm:main} may suggest that we do not use the full power of property (T), but rather the apparently weaker property that any weak$^*$ limit of \emph{weakly-mixing} measures in $M(\Gamma,\Omega)$ is ergodic. However, it is a result of Kechris \cite[Theorem 12.8]{MR2583950} that this property is equivalent to property (T), see also \cite{MR2458045}. 
\end{remark}


\begin{remark}\label{r.fiid}
Our proof strategy seems to break down if one wanted to prove that every infinite Kazhdan group has fixed price 1, or equivalently that $\cost^*(\Gamma)=1$ as defined in~(\ref{eq:upper_cost_def}). 

Section~\ref{ss.redu}, the reduction part, continues to work in the FIID setting: Indeed, if one can construct a FIID process in $M(\Gamma,\conn(\Gamma))$ with expected degree at most $\eps$, then either proof of \cref{prop:sparse_cost} will yield a process in $F_{\mathrm{IID}}(\Gamma,\spann(\Gamma))$ with expected degree at most $2+\eps$. (The fact that the WUSF is a FIID can be deduced from the `stack of arrows' implementation of Wilson's algorithm and its interpretation in terms of  \emph{cycle-popping} \cite{Wilson96,BLPS}.)

On the other hand, it seems unlikely that the thinning procedure in the construction of \cref{ss.constru} can be carried out using FIID processes. Indeed, as explained by Klaus Schmidt in the proof of Theorem 2.4 of \cite{Schmidt}, it was implicitly proved by Losert and Rindler \cite{LosertRindler} that the Markov operator for any generating set of a nonamenable group $\Gamma$ acting on $L^2([0,1]^\Gamma,\mathrm{Leb}^{\otimes\Gamma})$ has a spectral gap, and hence that the Bernoulli shift is strongly ergodic. See \cite[Section 3]{KechrisTsankov} and \cite[Theorem 3.1]{BSzV15} for related results.
 This spectral gap implies that the agreement probability for some pair of neighbours is separated away from 1 in any FIID site percolation of fixed density $p \in (0,1)$, and this bound is clearly inherited by weak$^*$ limits.  (More generally, it is a theorem of Ab\'ert and Weiss \cite[Theorem 4]{AbW} that any weak$^*$ limit of factors of a strongly ergodic process is ergodic.) 
Thus, by \cref{cor:degeneration}, on \emph{any} nonamenable Cayley graph there exists $i_{\mathrm{fiid}}<\infty$ such that $\mu_i$ is not FIID for $i > i_{\mathrm{fiid}}$. There seems to be no reason to expect that $i_{\mathrm{fiid}} =i_{\mathrm{freq}}$ in the Kazhdan case, which would be needed to prove $\cost^*(\Gamma)=1$ via this strategy.
\end{remark}

\begin{remark}
It is perhaps better to think of the proof of \cref{thm:main} as a proof of the \emph{contrapositive} of that theorem, i.e., as a proof that every countable group with cost $>1$ does not have property (T). Indeed, if $\Gamma$ is finitely generated with $\cost(\Gamma)>1$, then running our iterations with $p>0$ small enough we can never arrive at a unique infinite cluster, and hence we obtain an explicit sequence $\mu_i \in E(\Gamma,\Omega)$ converging to the non-ergodic measure $p \delta_\Gamma + (1-p)\delta_\emptyset$.
\end{remark}

\section*{Acknowledgments}

We are grateful to Mikl\'os Ab\'ert for many helpful discussions, to MFO, Oberwolfach, where this work was conceived, and to Damien Gaboriau and Russ Lyons for comments on the manuscript. We also thank the anonymous referees for their thorough reading and helpful comments. The work of GP was supported by the ERC Consolidator Grant 772466 ``NOISE'', and by the Hungarian National Research, Development and Innovation Office, NKFIH grant K109684.

  \setstretch{1}
  \bibliographystyle{abbrv}
  \bibliography{unimodularthesis}
\end{document}